\documentclass[preprint,12pt,3p]{elsarticle}
\usepackage{amsmath,amsthm,amsfonts,amssymb}
\usepackage{fullpage}
\usepackage{subfig}
\usepackage{blkarray}
\usepackage{caption}
\usepackage{float}
\usepackage{algpseudocode}
\usepackage{tikz}

\usepackage{hyperref}
\numberwithin{equation}{section}
\usepackage{xcolor}
\usepackage{colortbl}
\usepackage[normalem]{ulem}
\usepackage{sectsty}

\definecolor{astral}{RGB}{46,116,181}
\subsectionfont{\color{astral}}
\sectionfont{\color{astral}}
\usepackage{colortbl}

\DeclareMathAlphabet{\mathpzc}{OT1}{pzc}{m}{it}
\DeclareFontFamily{OT1}{pzc}{}
\DeclareFontShape{OT1}{pzc}{m}{it}{<-> s * [0.900] pzcmi7t}{}
\DeclareMathAlphabet{\mathpzc}{OT1}{pzc}{m}{it}

\usepackage{amsbsy}
\usepackage{amsmath}
\usepackage{accents}
\newlength{\dhatheight}

\newenvironment{frcseries}{\fontfamily{frc}\selectfont}{}

\DeclareMathAlphabet\mathbfcal{OMS}{cmsy}{b}{n}
\definecolor{darkslategray}{rgb}{0.18, 0.31, 0.31}
\definecolor{warmblack}{rgb}{0.0, 0.26, 0.26}

\usepackage{multirow}

\usepackage{mathtools}
\usepackage{algorithm}
\usepackage{algorithmicx}
\makeatletter
\def\BState{\State\hskip-\ALG@thistlm}
\makeatother

\newtheorem{theorem}{Theorem}[section]
\newtheorem{lemma}[theorem]{Lemma}

\theoremstyle{definition}
\usepackage{hyperref}
\newtheorem{definition}{Definition}[section]
\newtheorem{remark}{Remark}[section]
\newtheorem{example}{Example}[section]
\journal{}

\begin{document}

\begin{frontmatter}

\title{Reverse order law for NDMPI of dual matrices and its applications}

\vspace{-.4cm}

 \author{Tikesh Verma$^{1a}$, Amit Kumar$^{2b}$, Debasisha Mishra $^{3a}$ }
  \address{

                   $^a$Department of Mathematics,\\
                        National Institute of Technology Raipur, India.\\
                        $^b$Department of Mathematics\\
                        Galgotias University, Greater Noida, Uttar Pradesh, India.
                        \\email$^1$: rprtikesh@gmail.com\\  
                        email$^2$:amitdhull513@gmail.com\\
                        email$^2$:dmishra@nitrr.ac.in\\

                        }
\vspace{-2cm}

\begin{abstract}
This manuscript establishes several sufficient conditions for the validity of both the reverse order law and forward order law for NDMPI. Additionally, some characterization of the reverse order law of the  NDMPI is obtained.  We also explore the applications of the reverse order law within this framework. Finally, we demonstrate the additivity of the NDMPI, supported by illustrative examples.

\noindent

\end{abstract}


\end{frontmatter}

\newpage
\section{Introduction}

A {\it dual number} is represented as $ \hat{a} = a_s + a_d \epsilon $, where $a_s \in \mathbb{R} $ is the {\it standard} part of $\hat{a}$, $ a_d \in \mathbb{R}$ is the {\it dual} part of $\hat{a}$, and $ \epsilon $ is an infinitesimal unit such that $ \epsilon^2 = 0 $. Dual numbers were introduced by William Clifford \cite{cliford:1873} in 1873, and these numbers form a commutative ring. If $a_{s}$ and $a_{d}$ are complex numbers, then $\hat{a}$ is called a {\it dual complex number}.  The dual complex numbers are utilized to find numerical solutions of dual polynomial equations \cite{cheng:1996}. In addition, dual numbers and dual complex numbers have a wide range of applications in robotics, kinematic analysis, and rigid body motion (see \cite{fischer:2017, ulrich:2016, wang:2012}).\par
A matrix is called a {\it dual matrix (dual complex matrix)} if its entries are dual numbers (dual complex numbers).
Several researchers established the mathematical properties and applications of dual complex matrices. In 1987, Gu and Luh \cite{Gu:1987} considered the inverse of a dual complex matrix $\hat{A}=A_{s}+\hat{A}_{d}\epsilon \in \mathbb{D}\mathbb{C}^{n \times n}$ as $\hat{A}^{-1}=A_{s}^{-1}-A_{s}^{-1}A_{d}A_{s}^{-1}\epsilon$. Pennestrì and Stefanelli \cite{pennestri:2007} discussed the {\it Moore-Penrose dual generalized inverse (MPDGI)} of $\hat{A}=A_{s}+A_{d}\epsilon \in \mathbb{D}\mathbb{C}^{m \times n}$ and is defined as follows: $\hat{A}^{P}=A_{s}^{\dag}-A_{s}^{\dag}A_{d}A_{s}^{\dag}\epsilon$, where $A_{s}^{\dag}$ reprents the Moore-Penrose inverse of complex matrix $A_{s}$. In 2012, Angeles \cite{angeles:2012} studied the applications of the MPDGI in kinematic synthesis, followed by Pennestrì {\it et al.} \cite{pennestri:2018} in 2018. The MPDGI of a dual complex matrix always exists; however, it may not satisfy all the Moore-Penrose conditions. In 2018, De Falco {\it et al.} \cite{de:2018} discussed the {\it dual Moore-Penrose generalized inverse (DMPGI)}, which satisfies all the Moore-Penrose conditions. It is well known that the Moore-Penrose generalized inverse exists for all matrices, but the DMPGI  may not be exist. It was shown by Udwadia {\it et al.} \cite{udwadia:2020} that there are infinitely many dual matrices for which the DMPGI does not exist. Udwadia \cite{udwadia:2021} further obtained the equivalent conditions under which the DMPGI of dual matrices exist. In 2021, Wang \cite{wang:2021} also studied necessary and sufficient conditions for the existence of the DMPGI for dual matrices. Qi { \it et al.} \cite{qi:2022} presented the eigenvalue decomposition of the dual complex Hermitian matrix and the singular value decomposition of the dual complex matrix. Very recently, Cui and  Qi \cite{cui:2024} modified the first Moore-Penrose condition via the singular value decomposition of dual complex matrices and then demonstrates that every dual complex matrix has a unique {\it new dual Moore-Penrose inverse (NDMPI)}. They have also applied the NDMPI of dual complex matrices to obtain the minimum norm least squares solution for the dual linear system of equations. \par

The equality \((AB)^{-1} = B^{-1}A^{-1}\) is known as the reverse order law for for invertible matrices, while \((AB)^{-1} = A^{-1}B^{-1}\) represents the forward order law. The reverse order law always true for invertible matrices; however, the forward order law does not hold for invertible matrices. It's important to note that these laws do not hold for various generalized inverses of matrices. The reverse order law of generalized inverses plays a significant role in theoretical and numerical computations in areas such as singular matrix problems, ill-posed problems, and optimization problems (see \cite{ben:1974, golub:1996}). Many research articles focused on identifying sufficient or equivalent conditions under which both the reverse and forward order laws hold (see \cite{barwick:1974, castro:2018, greville:1966, shinozaki:1979}) for matrices. For example, In 1966, Greville \cite{greville:1966} first obtained some sufficient conditions under which the reverse order law holds for the Moore-Penrose inverse. Dinčić and Djordjević \cite{dincic:2013} extended the reverse order law for the Moore-Penrose inverse of bounded linear operators. Recently, Castro-González and Hartwig \cite{castro:2018} provided expressions for the reverse and forward order laws of the Moore-Penrose inverse of matrices.  But, the notion of reverse and forward order laws for dual complex matrices is not discussed yet. Motivated by this, we investigate the reverse and forward order law for the NDMPI of dual complex matrices.\par
 The main contributions of the manuscript are summed up as follows: The manuscript is organized as follows. In Section \ref{sec2}, we review some preliminary results for dual complex matrices. Section \ref{sec3} is divided into two subsections. In the first subsection, we present various sufficient conditions for the reverse and forward order laws of the NDMPI are established. In the second subsection, we present an application of the reverse order law to compute the NDMPI of a certain class of dual complex matrices.\par

\section{Preliminaries}{\label{sec2}}
In this section, we first briefly explain some of the terminologies used in this article. The set of real numbers, complex numbers, dual numbers, and dual complex numbers are denoted by $\mathbb{R}$, $\mathbb{C}$, $\mathbb{D}$, $\mathbb{D}\mathbb{C}$. A dual number is called appreciable if its standard part is nonzero; otherwise, it is called infinitesimal. We recall the total order $~``\leq"$ over $\mathbb{D}$ defined by Qi {\it et al.} \cite{qil:2022}. Suppose that $\hat{a}=a_{s}+a_{d}\epsilon, \hat{b}=b_{s}+b_{d}\epsilon \in \mathbb{D}$. Then, 
\begin{enumerate}
    \item[(i)] $\hat{a}=\hat{b}$ if $a_{s}=b_{s}$ and $a_{d}=b_{d}$,
    \item[(ii)]  $\hat{a}<\hat{b}$ if $a_{s}<b_{s}$, or$a_{d}<b_{d}$,
    \item[(iii)] $\hat{a} \leq \hat{b}$ if $a_{s}<b_{s}$, or$a_{d} \leq b_{d}$.
\end{enumerate}
In particular, a dual number $\hat{a}=a_{s}+a_{d}\epsilon \in \mathbb{D}$ is positive, nonnegative, negative and nonpositive if $\hat{a}>0$, $\hat{a} \geq$, $\hat{a}<0$, and $\hat{a} \leq 0$, respectively.  For a dual complex number $a=a_{s}+a_{d}\epsilon$, its conjugate is defined as $\hat{a}^*=\bar{a}_{s}+\bar{a}_{d}\epsilon$, where $\bar{a}_{s}$ and $\bar{a}_{d}$ represent the complex conjugates of $a_{s}$ and $a_{d}$, respectively. We denote the set of all complex matrices of order $m \times n$ by $\mathbb{C}^{m \times n}$. For two complex matrices $A_{s}$, $A_{d} \in \mathbb{C}^{m \times n}$, $\hat{A}=A_{s}+A_{d}\epsilon \in \mathbb{D}\mathbb{C}^{m \times n}$ represents a dual complex matrix, where $\mathbb{D}\mathbb{C}^{m \times n}$ is the set of $m \times n$ dual complex matrices. The rank of $A \in \mathbb{C}^{m \times n}$ is denoted by $rank(A)$. $A^*=(\bar{a}_{ij})^T$ and $\hat{A}^*=A_{s}^*+A_{d}^*\epsilon$ denote the conjugate transpose of a complex matrix and dual complex matrix, respectively. A dual complex matrix is Hermitian if $\hat{A}=\hat{A}^*$,  and unitary if $\hat{A}^*\hat{A}=\hat{A}\hat{A}^*=I_{n}$, where $I_{n}$ is $n \times n$ identity matrix. We begin this section by recalling the singular value decomposition of a dual complex matrix.

\begin{theorem} {\rm (\cite[Theorem 5.2]{qi:2022})}\\
    Let $\hat{A} \in \mathbb{D}\mathbb{C}^{m \times n}$. Then, there exist dual complex unitary matrices $\hat{U} \in \mathbb{D}\mathbb{C}^{m \times m}$ and $\hat{V} \in \mathbb{D}\mathbb{C}^{n \times n}$ such that  $$\hat{A}=\hat{U}\hat{\Sigma}\hat{V}^*=\hat{U}\begin{bmatrix}
        \hat{\Sigma_{t}} & O\\
        O & O
    \end{bmatrix}\hat{V}^*,$$
  where $ \hat{\Sigma_{t}} \in \mathbb{D}^{t \times t}$ is a dual diagonal matrix of the form $$\hat{\Sigma_{t}}=diag(\hat{\mu}_{1},\dots,\hat{\mu}_{r},\hat{\mu}_{r+1},\dots,\hat{\mu}_{t}),$$
  $r\leq t \leq \text{minimum}\{m,n\}$, $\hat{\mu}_{1} \geq \dots \geq \hat{\mu}_{r}$ are positive appreciable dual numbers and $\hat{\mu}_{r+1} \geq \dots \geq \hat{\mu}_{t}$ are positive infinitesimal dual numbers. Counting possible multiplicities of the diagonal entries, the form of $\hat{\Sigma_{t}}$ is unique. 
\end{theorem}
We can also write $$\hat{\Sigma}=\begin{bmatrix}
    \hat{\Sigma_{t}} & O\\
    O & O
\end{bmatrix}=\begin{bmatrix}
    \hat{\Sigma_{1}} & O\\
    O & \hat{\Sigma_{2}}
\end{bmatrix},$$
where $\hat{\Sigma_{1}} \in \mathbb{D}^{r \times r}$, $\hat{\Sigma_{1}}=diag(\hat{\mu}_{1},\dots,\hat{\mu}_{r})$ and $\hat{\Sigma_{2}}=diag(\hat{\mu}_{r+1},\dots,\hat{\mu}_{t},0,\dots,0)$. Further, the essential part and nonessential part of a matrix $\hat{A} \in \mathbb{D}\mathbb{C}^{m \times n}$ via the representation of $\hat{\Sigma}$  are defined as 
$$\hat{A}_{e}=\hat{U}\begin{bmatrix}
   \hat{\Sigma_{1}} & O\\
    O & O
\end{bmatrix}\hat{V}^*=\hat{U}\begin{bmatrix}
    \Sigma_{1s}+\Sigma_{1d}\epsilon & O\\
    O & O
\end{bmatrix}\hat{V}^*~\text{and}~\hat{A}_{n}=\hat{U}\begin{bmatrix}
    O & O\\
    O & \hat{\Sigma_{2}}
\end{bmatrix}\hat{V}^*=\hat{U}\begin{bmatrix}
    O & O\\
    O & \Sigma_{2d}
\end{bmatrix}\hat{V}^*\epsilon.$$
The Moore–Penrose inverse of a matrix celebrated its centenary in 2020. The manuscript by E H Moore \cite{moore:1920} that provides the first definition of the notion the Moore–Penrose inverse
 did not attract much attention. After thirty five years later, unaware
 of the work done by Moore, Sir Roger Penrose \cite{penrose:1955} provided an equivalent definition of the same concept. Richard Rado \cite{rado:1956} recognized that both definitions referred to the same concept in 1956. The Penrose conditions, consisting of four matrix equations, outline the definition of the new dual Moore-Penrose inverse of a dual complex matrix below. 
\begin{definition}{\rm (\cite[Definition 4.1]{cui:2024})}\label{def2.1}\\
    Let $\hat{A} \in \mathbb{D}\mathbb{C}^{m \times n}$. Then, the matrix $\hat{X} \in \mathbb{D}\mathbb{C}^{n \times m}$ is called the NDMPI of $\hat{A}$ if it satisfies the following four equations:
    \begin{align}
     \hat{A}\hat{X}\hat{A}&=\hat{A}_{e},\label{ptv2.1} \\
     \hat{X}\hat{A}\hat{X}&=\hat{X},  \label{pt2.2}\\
     (\hat{A}\hat{X})^*&=\hat{A}\hat{X},\label{pt2.3} \\ 
     (\hat{X}\hat{A})^*&=\hat{X}\hat{A}. \label{ptv2.4}
     \end{align} It is denoted by $\hat{A}^N$.
\end{definition}
Li and Wang \cite{Li:2023} replaced the first Moore-Penrose condition with 
\begin{equation}\label{ptv2.5}
    \hat{A}^*\hat{A}\hat{X}\hat{A}\hat{A}^*=\hat{A}^*\hat{A}\hat{A}^*
\end{equation} and defined weakly dual MP generalized inverse as
$$\hat{X}=(\hat{A}^*\hat{A})^\dag\hat{A}^*=\hat{A}^*(\hat{A}\hat{A}^*)^\dag.$$
 For any $\hat{A} \in \mathbb{D}\mathbb{C}^{m \times n}$, the set of all dual complex matrices $\hat{X}\in\mathbb{D}\mathbb{C}^{n\times m}$ which satisfies any of the combinations of the above four dual complex matrix equations \eqref{ptv2.1}-\eqref{ptv2.4} is denoted as $\hat{A}\{i,j,k,l\}$, where $i,j,k,l \in \{1,2,3,4\}$. For example, $\hat{A}\{1\}$ denotes the set of all solutions $\hat{X}$ of dual complex matrix equation (1). Such an $\hat{X}$ which satisfies equation (1) is called an inner inverse of $\hat{A}$, and is denoted by  $\hat{A}^{(1)}$. Similarly, $\hat{A}\{1,3\}$ denotes the set of all solutions of the first and third dual complex matrix equations. We denote a member of $\hat{A}\{1,3\}$ as $\hat{A}^{(1,3)}$.

The next result shows that conditions \eqref{ptv2.1} and \eqref{ptv2.5} are equivalent.

\begin{theorem}{\rm (\cite[Theorem 4.3]{cui:2024})}\label{th2.2}
    Let $\hat{A} \in \mathbb{D}\mathbb{C}^{m \times n}$. Then, under conditions \eqref{pt2.3} and \eqref{ptv2.4}, the following are equivalent:
     \begin{enumerate}[(i)]
         \item $\hat{A}^*\hat{A}\hat{X}\hat{A}\hat{A}^*=\hat{A}^*\hat{A}\hat{A}^*$,
         \item $\hat{A}\hat{X}\hat{A}=\hat{A}_{e}$.
     \end{enumerate}
    
\end{theorem}
\begin{remark}\label{re2.1}
The above result can also be proved by taking $(\hat{A}_{e}\hat{X})^*=\hat{A}_{e}\hat{X}$ and $(\hat{X}\hat{A}_{e})^*=\hat{X}\hat{A}_{e}$. 
\end{remark}The following result provides a method for finding the NDMPI.
\begin{theorem} {\rm (\cite[Theorem 4.1]{cui:2024})}\\
    Let $\hat{A} \in \mathbb{D}\mathbb{C}^{m \times n}$ and  $$\hat{A}=\hat{U}\begin{bmatrix}
        \hat{\Sigma_{1}} & O\\
        O & \hat{\Sigma_{2}}
    \end{bmatrix}\hat{V}^*,$$ be the SVD of $\hat{A}$. Then, the NDMPI of $\hat{A}$ is 
    $$\hat{A}^N=\hat{V}\begin{bmatrix}
        \hat{\Sigma_{1}}^{-1} & O\\
        O & O
    \end{bmatrix}\hat{U}^*.$$
\end{theorem}
Now, we recall some properties of the NDMPI.
\begin{theorem}{\rm (\cite[Theorem 3.2]{be:2024})}\label{ptv2.3}
    Let $\hat{A} \in \mathbb{D}\mathbb{C}^{m \times n}$ . Then, 
    \begin{enumerate}
        \item $\hat{A}^*_{e}=\hat{A}^*A\hat{A}^N=\hat{A}^N\hat{A}\hat{A}^*$.
        \item $\hat{A}_{e}=(\hat{A}^*)^N\hat{A}^*\hat{A}=\hat{A}\hat{A}^*(\hat{A}^*)^N$.
        \item $\hat{A}^N=(\hat{A}^*\hat{A})^{N}\hat{A}^*=\hat{A}^*(\hat{A}\hat{A}^*)^N$.
    \end{enumerate}
\end{theorem}

Wang \cite{wang:2023} introduced the Definition of $r$-column full rank dual complex matrix, $r$-row full rank dual complex matrix and dual $r$-rank decomposition of a dual complex matrix. The same is recalled in the following definitions.

\begin{definition}{\rm (\cite[Definition 1]{wang:2023})}
    Let $\hat{A} = A_{s} + A_{d}\epsilon \in  \mathbb{D}\mathbb{C}^{m \times r}$, $\hat{B} = B_{s}+B_{d}\epsilon \in  \mathbb{D}\mathbb{C}^{r \times n}$. If the standard part $A_{s}$ of $\hat{A}$ is a full column rank matrix, then we call $\hat{A}$ $r$-column full
rank dual complex matrix; if the standard part $B_{s}$ of $\hat{B}$ is a full row rank matrix, then we call $\hat{B}$ $r$-row full rank dual complex matrix.
\end{definition}

\begin{definition}{\rm (\cite[Definition 2]{wang:2023})}
    Let $\hat{A}=A_{s}+A_{d}\epsilon \in \mathbb{D}\mathbb{C}^{m \times n}$, $rank(A_{s})=r$ and $A_{s}=A_{1s}A_{2s}$ be full rank decomposition of $A_{s}$. If there exist, an $r$-column full rank dual complex matrix $\hat{A}_{1}=A_{1s}+A_{1d}\epsilon$ and $r$-row full rank dual complex matrix $\hat{A}_{2}=A_{2s}+A_{2d}\epsilon$, such that  $$\hat{A}=\hat{A}_{1}\hat{A}_{2},$$
    which is called dual $r$-rank decomposition of $\hat{A}$.
\end{definition}

\section{Main Results}\label{sec3}

In this section, we provide several sufficient conditions for the reverse order law of the NDMPI of a dual complex matrix. Using the singular value decomposition of a dual complex matrix, we obtain the following lemma.

\begin{lemma}\label{mptv4.1}
    Let $\hat{A}, \hat{B} \in \mathbb{D}\mathbb{C}^{m \times n}$. Then, 
    \begin{enumerate}[(i)]
        \item $\hat{A}^*\hat{A}=\hat{A}^*_{e}\hat{A}=\hat{A}^{*}\hat{A}_{e}$.
        \item $\hat{A}\hat{A}^*=\hat{A}_{e}\hat{A}^*=\hat{A}\hat{A}^*_{e}$.
        \item $\hat{A}^N\hat{A}=\hat{A}^N\hat{A}_{e}$.
        \item $\hat{A}_{e}\hat{A}^N=\hat{A}\hat{A}^N$.
        \item If $\hat{B}$ is invertible, then $\hat{A}^*\hat{A}\hat{B}=\hat{A}^*(\hat{A}\hat{B})_{e}$ and  $\hat{A}_{e}\hat{B}=\hat{A}\hat{A}^N(\hat{A}\hat{B})_{e}$.
    \end{enumerate}
\end{lemma}
\begin{proof}
   We will show (i), (iii) and (v). The rest assertions are  similar to other ones. Let $\hat{A}=\hat{U}\begin{bmatrix}
        \hat{\Sigma_{1}} & O\\
        O & \hat{\Sigma_{2}}
    \end{bmatrix}\hat{V}^*$ be the SVD of $\hat{A}$. Then, 
    $$\hat{A}_{e}=\hat{U}\begin{bmatrix}
        \hat{\Sigma_{1}} & O\\
        O & O
    \end{bmatrix}\hat{V}^*.$$
        (i) $\hat{A}^*\hat{A}=\hat{V}\begin{bmatrix}
        \hat{\Sigma_{1}^*} & O\\
        O & {\Sigma_{2d}^*}\epsilon
    \end{bmatrix}\hat{U}^*\hat{U}\begin{bmatrix}
        \hat{\Sigma_{1}} & O\\
        O & \Sigma_{2d}\epsilon
    \end{bmatrix}\hat{V}^*=\hat{V}\begin{bmatrix}
        \hat{\Sigma_{1}}^2 & O\\
        O & O
    \end{bmatrix}\hat{V}^*=\hat{A}^*_{e}\hat{A}=\hat{A}^*\hat{A}_{e}.$\\
    
    (iii) By the Definition of the NDMPI, we have
    $$\hat{A}^N\hat{A}=\hat{V}\begin{bmatrix}
        \hat{\Sigma^{-1}_{1}} & O\\
        O & O
    \end{bmatrix}\hat{U}^*\hat{U}\begin{bmatrix}
        \hat{\Sigma_{1}} & O\\
        O & \hat{\Sigma_{2}}
    \end{bmatrix}\hat{V}^*=\hat{V}\begin{bmatrix}
        I & O\\
        O & O
    \end{bmatrix}\hat{V}^*=\hat{A}^N\hat{A}_{e}.$$
    (v) Using (i), we have  $(\hat{A}\hat{B})^*(\hat{A}\hat{B})=(\hat{A}\hat{B})^*(\hat{A}\hat{B})_{e}$. Since $\hat{B}$ is invertible, we get 
    \begin{equation}\label{eq3.1}
     \hat{A}^*\hat{A}\hat{B}=\hat{A}^*(\hat{A}\hat{B})_{e}.
     \end{equation}
     Pre-multiplying \eqref{eq3.1} by $(A^*)^N$ and using Theorem \ref{ptv2.3}, we get 
$$\hat{A}_{e}\hat{B}=\hat{A}\hat{A}^N(\hat{A}\hat{B})_{e}.$$
    
\end{proof}
 
Next, we will present a result that is essential for proving the main result of this section.

  \begin{theorem} \label{mptv4.6}
      Let $\hat{A} \in \mathbb{D}\mathbb{C}^{m \times q}$ and $\hat{B} \in \mathbb{D}\mathbb{C}^{q \times n}$ be dual matrices. If $\hat{A}^N\hat{A}\hat{B}\hat{B}^*\hat{A}^*=\hat{B}\hat{B}^*\hat{A}^*_{e}$, $\hat{B}\hat{B}^N\hat{A}^*\hat{A}\hat{B}=\hat{A}^*\hat{A}\hat{B}_{e}$, then \begin{enumerate}
       \item $\hat{A}^N\hat{A}\hat{B}\hat{B}^*=\hat{B}\hat{B}^*\hat{A}^N\hat{A}$.
       \item $\hat{A}_{e}\hat{B}\hat{B}^*\hat{A}^*=\hat{A}\hat{B}\hat{B}^*\hat{A}^*_{e}$.
       \item $\hat{B}\hat{B}^N\hat{A}^*\hat{A}=\hat{A}^*\hat{A}\hat{B}^N\hat{B}$.

          \item $\hat{B}^*_{e}\hat{A}^*\hat{A}\hat{B}=\hat{B}^*\hat{A}^*\hat{A}\hat{B}_{e}.$
         
          \end{enumerate}
  \end{theorem}
  \begin{proof}
      Suppose that $\hat{A}^N\hat{A}\hat{B}\hat{B}^*\hat{A}^*=\hat{B}\hat{B}^*\hat{A}^*_{e}$, which on post-multiplying by $(\hat{A}^*)^N$ gives 
      \begin{equation}\label{mptv4.7}
          \hat{A}^N\hat{A}\hat{B}\hat{B}^*\hat{A}^N\hat{A}=\hat{B}\hat{B}^*\hat{A}^N\hat{A}.
\end{equation}
      Since the left side of \eqref{mptv4.7} is Hermitian, we have \begin{equation}\label{eq3.2}
      \hat{A}^N\hat{A}\hat{B}\hat{B}^*=\hat{B}\hat{B}^*\hat{A}^N\hat{A}.
          \end{equation}
       Recall that ${A}^N\hat{A}\hat{A}^*=A_{e}^*$. Pre and post-multiplying \eqref{eq3.2} by $\hat{A}$ and $\hat{A}^*$, we get 
       $$\hat{A}_{e}\hat{B}\hat{B}^*\hat{A}^*=\hat{A}\hat{B}\hat{B}^*\hat{A}^*_{e}.$$
       Similarly, one can prove $\hat{B}\hat{B}^N\hat{A}^*\hat{A}=\hat{A}^*\hat{A}\hat{B}^N\hat{B}$ and $\hat{B}^*_{e}\hat{A}^*\hat{A}\hat{B}=\hat{B}^*\hat{A}^*\hat{A}\hat{B}_{e}.$
  \end{proof}

The main result of this section provided below.

  \begin{theorem}\label{mtv3.5}
       Let $\hat{A} \in \mathbb{D}\mathbb{C}^{m \times q}$ and $\hat{B} \in \mathbb{D}\mathbb{C}^{q \times n}$ be dual complex matrices. If 
      \begin{equation}\label{mptv4.8}
          \hat{A}^N\hat{A}\hat{B}\hat{B}^*\hat{A}^*=\hat{B}\hat{B}^*\hat{A}^*_{e} 
           \end{equation} and \begin{equation}\label{mptv4.9}
               \hat{B}\hat{B}^N\hat{A}^*\hat{A}\hat{B}=\hat{A}^*\hat{A}\hat{B}_{e},
               \end{equation} then $(\hat{A}\hat{B})^N=\hat{B}^N\hat{A}^N$.
         \end{theorem}

       In the next result, we modify conditions \eqref{mtv3.5} and \eqref{mptv4.8}.
      \begin{theorem}\label{mtvth3.5}
           Let $\hat{A} \in \mathbb{D}\mathbb{C}^{m \times q}$ and $\hat{B} \in \mathbb{D}\mathbb{C}^{q \times n}$ be dual complex matrices. If 
      \begin{equation}\label{mptv3.11}
          \hat{A}^N\hat{A}\hat{B}\hat{B}^*\hat{A}^*=\hat{B}\hat{B}^*\hat{A}^* 
           \end{equation} and \begin{equation}\label{mptv3.12}
               \hat{B}\hat{B}^N\hat{A}^*\hat{A}\hat{B}=\hat{A}^*\hat{A}\hat{B},
               \end{equation} then $(\hat{A}\hat{B})^N=\hat{B}^N\hat{A}^N$. 
      \end{theorem}

  The converse of the above theorem is not true. An example illustrating is provided below.
     \begin{example}
          Let $\hat{A}=\begin{bmatrix}
                  1 & \epsilon 
          \end{bmatrix}$ and $\hat{B}=\begin{bmatrix}
              1 & 0\\
              0 & \epsilon
          \end{bmatrix}$. Then, 
          $(\hat{A}\hat{B})^N=\begin{bmatrix}
              1\\
              0
          \end{bmatrix}$ and 
             $\hat{B}^N\hat{A}^N=\begin{bmatrix}
              1 \\
              0 
          \end{bmatrix}$. Thus, $(\hat{A}\hat{B})^N=\hat{B}^N\hat{A}^N$ but $\hat{B}\hat{B}^N\hat{A}^*\hat{A}\hat{B}=\begin{bmatrix}
              1 & 0\\
              0 & 0
          \end{bmatrix} \neq \begin{bmatrix}
              1 & 0\\
              \epsilon & 0
          \end{bmatrix}=\hat{A}^*\hat{A}\hat{B}.$
      \end{example}
We now provide a single condition for the reverse order law instead of two.
  \begin{theorem}
       Let $\hat{A} \in \mathbb{D}\mathbb{C}^{m \times q}$ and $\hat{B} \in \mathbb{D}\mathbb{C}^{q \times n}$ be dual complex matrices. If 
       \begin{equation}\label{mptv3.10}
           \hat{A}^N\hat{A}\hat{B}\hat{B}^*\hat{A}^*\hat{A}\hat{B}\hat{B}^N=\hat{B}\hat{B}^*\hat{A}^*\hat{A},
       \end{equation}   then $(\hat{A}\hat{B})^N=\hat{B}^N\hat{A}^N$.
  \end{theorem}
  \begin{proof}
      Pre-multiplying \eqref{mptv3.10} by $\hat{A}^N\hat{A}$ gives \begin{equation}\label{mptv3.11}
          \hat{A}^N\hat{A}\hat{B}\hat{B}^*\hat{A}^*\hat{A}\hat{B}\hat{B}^N=\hat{A}^N\hat{A}\hat{B}\hat{B}^*\hat{A}^*\hat{A}.
      \end{equation}
      From \eqref{mptv3.10} and \eqref{mptv3.11}, we have $\hat{A}^N\hat{A}\hat{B}\hat{B}^*\hat{A}^*\hat{A}=\hat{B}\hat{B}^*\hat{A}^*\hat{A},$
      which on post-multiplying  by $\hat{A}^N$, we get   $$\hat{A}^N\hat{A}\hat{B}\hat{B}^*\hat{A}_{e}^*=\hat{B}\hat{B}^*\hat{A}^*_{e},$$
      which in turn implies 
      $$\hat{A}^N\hat{A}\hat{B}\hat{B}^*\hat{A}^*=\hat{B}\hat{B}^*\hat{A}^*_{e}$$
      by following the technique employed in the converse part of Theorem 3.4. Similarly, we can prove $$\hat{B}\hat{B}^N\hat{A}^*\hat{A}\hat{B}=\hat{A}^*\hat{A}\hat{B}_{e}.$$
      Thus, by Theorem \ref{mtv3.5}, we obtain $(\hat{A}\hat{B})^N=\hat{B}^N\hat{A}^N.$

  \end{proof}
   The converse of the above theorem is not true. It is illustrated by the following example.
  \begin{example}
      Consider the same example as in \ref{ex4.2}. Then, $\hat{A}^N\hat{A}\hat{B}\hat{B}^*\hat{A}^*\hat{A}\hat{B}\hat{B}^N=\begin{bmatrix}
          2 & \epsilon\\
          2\epsilon & 0
      \end{bmatrix} \neq \begin{bmatrix}
          2 & 2\epsilon\\
          \epsilon & 0
      \end{bmatrix}=\hat{B}\hat{B}^*\hat{A}^*\hat{A}$.
  \end{example}
  The relationship between the DMPGI and the NDMPI of the product of dual complex matrices is given by the following result.
  \begin{theorem}
      Let $\hat{A} \in \mathbb{D}\mathbb{C}^{m \times q}$ and $\hat{B} \in \mathbb{D}\mathbb{C}^{q \times n}$ be dual complex matrices such that $(\hat{A}\hat{B})_{e}=(\hat{A}\hat{B})$. If \begin{align}\label{mptv3.13}
          \hat{A}^N\hat{A}\hat{B}_{e}&=\hat{B}(\hat{A}\hat{B})^N\hat{A}\hat{B},\\ \label{mptv3.14}
          \hat{B}\hat{B}^N\hat{A}^*_{e}&=\hat{A}^*\hat{A}\hat{B}(\hat{A}\hat{B})^N,
      \end{align} then  $(\hat{A}\hat{B})^\dag=(\hat{A}\hat{B})^N=\hat{B}^N\hat{A}^N$.
  \end{theorem}
  \begin{proof}
      Post-multiplying \eqref{mptv3.13} by $(\hat{A}\hat{B})^*$, we get
            $$\hat{A}^N\hat{A}\hat{B}\hat{B}^*\hat{A}^*=\hat{B}\hat{B}^*\hat{A}^*.$$  
       
      Now, post-multiplying \eqref{mptv3.14} by $(\hat{A}\hat{B})$, we have
     $$\hat{B}\hat{B}^N\hat{A}^*\hat{A}\hat{B}=\hat{A}^*\hat{A}\hat{B}.$$
     Thus, by Theorem \ref{mtvth3.5}, we get $(\hat{A}\hat{B})^\dag=(\hat{A}\hat{B})^N=\hat{B}^N\hat{A}^N$.
  \end{proof}

  In the following theorem, we obtain another set of conditions for the reverse order law. 
  \begin{theorem}
       Let $\hat{A} \in \mathbb{D}\mathbb{C}^{m \times q}$ and $\hat{B} \in \mathbb{D}\mathbb{C}^{q \times n}$ be dual complex matrices. If 
 \begin{align}\label{mptv4.17}
     \hat{A}^N\hat{A}\hat{B}\hat{B}^N&=\hat{B}\hat{B}^N\hat{A}^N\hat{A},\\ \label{mptv4.18}  
     (\hat{A}\hat{B}\hat{B}^N\hat{A}^N)^*&=\hat{A}\hat{B}\hat{B}^N\hat{A}^N,\\ \label{mptv4.19}
     (\hat{B}^N\hat{A}^N\hat{A}\hat{B})^*&=\hat{B}^N\hat{A}^N\hat{A}\hat{B},
 \end{align}
 then $(\hat{A}\hat{B})^N=\hat{B}^N\hat{A}^N$.
 
   \end{theorem}
  \begin{proof}
      We will show that $(\hat{A}\hat{B})^N=\hat{B}^N\hat{A}^N$ by the Definition of the NDMPI. The second and third assumptions confirm that the both $\hat{A}\hat{B}\hat{B}^N\hat{A}^N$ and $\hat{B}^N\hat{A}^N\hat{A}\hat{B}$ are Hermitian.
      
     Now, pre and post-multiplying \eqref{mptv4.17} by $(\hat{A}\hat{B})^*\hat{A}$ and $\hat{B}(\hat{B}\hat{A})^*$, we get
      $(\hat{A}\hat{B})^*\hat{A}\hat{B}\hat{B}^N\hat{A}^N\hat{A}\hat{B}(\hat{A}\hat{B})^*=(\hat{A}\hat{B})^*\hat{A}\hat{B}(\hat{A}\hat{B})^*,$
      which is equivalent to $$\hat{A}\hat{B}\hat{B}^N\hat{A}^N\hat{A}\hat{B}=(\hat{A}\hat{B})_{e}$$ 
      by Theorem \ref{th2.2}.
      Again, pre and post-multiplying \eqref{mptv4.17} by $\hat{B}^N$ and $\hat{A}^N$, we obtain 
$$\hat{B}^N\hat{A}^N\hat{A}\hat{B}\hat{B}^N\hat{A}^N=\hat{B}^N\hat{A}^N.$$
Thus, we have $$(\hat{A}\hat{B})^N=\hat{B}^N\hat{A}^N.$$
  \end{proof} 
 
    Next, we give a property of $\{1,2,3\}$-inverse of $\hat{A}$.
   
    \begin{lemma}\label{mptv3.18}
    Let $\hat{X} \in \hat{A}\{1,2,3\}$ for any $\hat{A} \in \mathbb{D}\mathbb{C}^{m \times n}$, then $\hat{A}\hat{X}=\hat{A}\hat{A}^N$.
\end{lemma}
\begin{proof}
    Post-multiplying $\hat{A}\hat{A}^N\hat{A}=\hat{A}_{e}$ by $\hat{X}$, we get 
    \begin{equation}\label{mptv3.22}  
    \hat{A}\hat{X}=\hat{A}\hat{A}^N\hat{A}\hat{X}
    \end{equation} 
    as $\hat{A}\hat{X}\hat{A}=\hat{A}_{e}$ and $\hat{X}\hat{A}\hat{X}=\hat{X}$ implies $\hat{A}\hat{X}=\hat{A}_{e}\hat{X}$. Now, taking the conjugate transpose of equation \eqref{mptv3.22}, we obtain $\hat{A}\hat{X}=\hat{A}\hat{A}^N$.
\end{proof}
Similarly, $\hat{X}\hat{A}=\hat{A}^N\hat{A}$ for $\hat{X} \in \hat{A}\{1,2,4\}$. Next, we present equivalent conditions for the reverse order law using $\{1,2,3\}$-inverse.
\begin{theorem}
    Let $\hat{A} \in \mathbb{D}\mathbb{C}^{m \times q}$ and $\hat{B} \in \mathbb{D}\mathbb{C}^{q \times n}$. Then, $(\hat{A}\hat{B})^N=\hat{B}^N\hat{A}^N$ if and only if there exists a $\{1,2,3\}$-inverse $\hat{X}$ of $\hat{B}$ satisfying $(\hat{A}\hat{B})^N=\hat{B}^N\hat{B}\hat{X}\hat{A}^N$. 
\end{theorem}

\begin{proof}
    Suppose that there exists a $\{1,2,3\}$-inverse $\hat{X}$ of $\hat{B}$ satisfying $(\hat{A}\hat{B})^N=\hat{B}^N\hat{B}\hat{X}\hat{A}^N$. Then, by Lemma \ref{mptv3.18}, we have $(\hat{A}\hat{B})^N=\hat{B}^N\hat{A}^N$. Conversely, let $X=\hat{B}^N$. Then, $\hat{X}$ is $\{1,2,3\}$-inverse of $\hat{B}$ and $(\hat{A}\hat{B})^N=\hat{B}^N\hat{A}^N=\hat{B}^N\hat{B}\hat{B}^N\hat{A}^N=\hat{B}^N\hat{B}\hat{X}\hat{A}^N$.
\end{proof}
Analogously, one can prove that $(\hat{A}\hat{B})^N=\hat{B}^N\hat{A}^N$ if and only if there exists a $\{1,2,4\}$-inverse $X$ of $\hat{A}$ satisfying $(\hat{A}\hat{B})^N=\hat{B}^N\hat{X}\hat{A}\hat{A}^N$.

Next, we provide some sufficient conditions under which the forward order law holds for the NDMPI of dual complex matrices.
  \begin{theorem}
      Let $\hat{A},\hat{B} \in \mathbb{D}\mathbb{C}^{m \times m}$ be dual complex matrices. If \begin{align}\label{mptv4.20}
          \hat{A}^*\hat{A}(\hat{A}\hat{B})^N\hat{B}\hat{B}^*&=\hat{A}^*\hat{B}^*,\\ \label{mptv4.21}
          \hat{B}\hat{B}^N\hat{A}\hat{B}&=\hat{A}\hat{B},\\   \label{mptv4.22}
          \hat{A}^N\hat{A}\hat{B}^*\hat{A}^*&=\hat{B}^*\hat{A}^*,
          \end{align}
          then $(\hat{A}\hat{B})^N=\hat{A}^N\hat{B}^N$.
  \end{theorem}

  The next example shows that the converse of the above theorem is not true.
  \begin{example}
      Let $\hat{A}=\begin{bmatrix}
          1 & 0\\
          0 & \epsilon
      \end{bmatrix}$ and $\hat{B}=\begin{bmatrix}
          1 & 0\\
          0 & 1
      \end{bmatrix}$. Then, $(\hat{A}\hat{B})^N=\begin{bmatrix}
          1 & 0\\
          0 & 0
      \end{bmatrix}$ and $\hat{A}^N\hat{B}^N=\begin{bmatrix}
          1 & 0\\
          0 & 0
      \end{bmatrix}$. Thus, $(\hat{A}\hat{B})^N=\hat{A}^N\hat{B}^N$ but $\hat{A}^N\hat{A}\hat{B}^*\hat{A}^*=\begin{bmatrix}
          1 & 0\\
          0 & 0
      \end{bmatrix} \neq \begin{bmatrix}
          1 & 0\\
          0 & \epsilon
      \end{bmatrix}=\hat{B}^*\hat{A}^*$.
  \end{example}
  The following result presents a set of sufficient conditions for the equation \((\hat{A}\hat{B})^N = \hat{A}^{-1}\hat{B}^N\) to hold true.\\
 
  \begin{theorem}
      Let $\hat{A},\hat{B} \in \mathbb{D}\mathbb{C}^{m \times m}$ be dual complex matrices such that $\hat{A}$ is invertible. If 
      \begin{align}\label{mptv4.24}
          \hat{B}^N(\hat{A}\hat{B})\hat{B}^*&=\hat{A}\hat{B}^*,\\ \label{mptv4.25}
          \hat{B}\hat{B}^N\hat{A}\hat{B}&=\hat{A}\hat{B},\\ \label{mptv4.26}
          \hat{A}\hat{B}(\hat{A}\hat{B})^N\hat{B}&=\hat{B},
      \end{align} then $(\hat{A}\hat{B})^N=\hat{A}^{-1}\hat{B}^N$.
  \end{theorem} 
  \begin{proof}
      Pre and post-multiplying \eqref{mptv4.24} by $\hat{A}^{-1}$ and $\hat{A}^*$, we get 
      \begin{equation}\label{eq3.32}
          (\hat{A}\hat{B})^*=\hat{A}^{-1}\hat{B}^N\hat{A}\hat{B}(\hat{A}\hat{B})^*.
       \end{equation}   
      Post-multiplying \eqref{eq3.32} by $(\hat{A}\hat{B}(\hat{A}\hat{B})^*)^N$ and using Theorem \ref{ptv2.3}, we have 
     \begin{equation}\label{eq3.23} (\hat{A}\hat{B})^N=\hat{A}^{-1}\hat{B}^N\hat{A}\hat{B}(\hat{A}\hat{B})^N.    
        \end{equation}   
   Again, post-multiplying \eqref{eq3.23} by $\hat{B}\hat{B}^N$ and using equation \eqref{mptv4.25} and \eqref{mptv4.26}, we get
        \begin{equation}\label{eq3.24}(\hat{A}\hat{B})^N=\hat{A}^{-1}\hat{B}^N.
        \end{equation}
    
  \end{proof}
     
   \section*{Data Availability Statements}

 Data sharing is not applicable to this manuscript as no datasets were generated or analyzed during the current study.

 \section*{Conflicts of interest}

 The authors declare that they have no conflict of interest.

 \section*{Acknowledgements}

The first author acknowledges the support of the National Institute of Technology Raipur, India.


\begin{thebibliography}{10}
\bibitem{angeles:2012}{Angeles, J., \textit{The Dual Generalized Inverses and Their Applications in Kinematic Synthesis,} in Latest Advances in Robot Kinematics, Springer Netherlands, Dordrecht, 2012, 1-10.}
\bibitem{barwick:1974}{Barwick, D.T.; Gilbert, J.D., \textit{On generalizations of the reverse order law,} SIAM J. Appl. Math., 27 (1974), 326-330.}
\bibitem{be:2024}{Be, A.; Mishra, D., \textit{Numerical range and numerical radius of dual complex matrices,} submitted.}
\bibitem{ben:1974}{Ben-Israel, A.; Greville, T.N.E.,} \textit{Generalized Inverses: Theory and Applications,} { Wiley, New York, 1974.}
\bibitem{castro:2018}{Castro-González, N.; Hartwig, R.E., \textit{Perturbation results and the forward order law for the Moore-Penrose inverse of a product,} Electron. J. Linear Algebra, 34 (2018), 514-525.}
\bibitem{cheng:1996}{Cheng, H.H.; Thompson, S., \textit{Dual polynomials and complex dual numbers for analysis of spatial
mechanisms,} In: Proceedings of the ASME Design Engineering Technical Conference and Computers
in Engineering Conference, Irvine, California, (1996), 18-22.}
\bibitem{cliford:1873}{Clifford, W.K., Preliminary sketch of bi-quaternions, 
	{\it Proc. Lond. Math. Soc.,} 4 (1873), 381-395.}
 \bibitem{cui:2024}{Cui, C.; Qi, L., \textit{A genuine extension of the Moore-Penrose inverse to dual matrices,} J. Comput. Appl. Math., 454 (2025), 116185.}
 \bibitem{de:2018}{De Falco, D.; Pennestrì, E.; Udwadia, F.E., \textit{On generalized inverses of dual matrices,} Mech. Mach. Theory, 123 (2018), 89-106.}
 \bibitem{dincic:2013}{Dinčić, N.Č.; Djordjević, D.S., \textit{Basic reverse order law and its equivalencies,} Aequationes Math., 85 (2013), 505-517.}
 \bibitem{fischer:2017}{Fischer, I.S., {\it Dual-Number Methods in Kinematics, Statics and Dynamics,} Routledge, 2017.}
 \bibitem{golub:1996}{Golub, G.H.; Van Loan, C.F., \textit{Matrix Computations,} Fourth edition, Johns Hopkins University Press, Baltimore, MD, 2013.}
 \bibitem{greville:1966}{Greville, T.N.E., \textit{Note on the generalized inverse of a matrix product,} SIAM Rev., 8 (1966), 518-521.}
 \bibitem{Gu:1987}{Gu, Y.L.; Luh, J., \textit{Dual-number transformation and its applications to robotics,} IEEE J.
Robotics Automat., 6 (1987), 615-623.}
 \bibitem{Li:2023}{Li, H.; Wang, H., \textit{Weak dual generalized inverse of a dual matrix and its
applications,} Heliyon, 9 (2023). https://doi.org/10.1016/j.heliyon.2023.e16624}
\bibitem{moore:1920}{Moore, E.H., \textit{On the reciprocal of the general algebraic matrix,} Bull. Amer. Math. Soc., 26 (1920), 394-395.}
\bibitem{pennestri:2007}{Pennestrì, E.; Stefanelli, R., {\it Linear algebra and numerical algorithms using dual numbers,} Multibody Syst. Dyn., 18 (2007), 323-344.}
\bibitem{pennestri:2018}{Pennestrì, E.; Valentini, P.P.; De Falco, D., \textit{The Moore–Penrose dual generalized inverse matrix with application to kinematic synthesis of spatial linkages,} J. Mech. Des., 140 (2018), 1-7.}
\bibitem{penrose:1955}{Penrose, R., \textit{A generalized inverse for matrices,} Proc. Cambridge Philos. Soc., 51 (1955), 406-413.}
\bibitem{qil:2022}{Qi, L.; Ling, C.; Yan, H., \textit{Dual quaternions and dual quaternion vectors,} Commun. Appl. Math. Comput., 4 (2022), 1494-1508.}
\bibitem{qi:2022}{Qi, L.; Alexander, D.M.; Chen, Z.; Ling, C.; Luo, Z., \textit{Low rank approximation of dual complex matrices,} (2022), arXIv:2201.12781v1.}
\bibitem{rado:1956}{Rado, R., \textit{Note on generalized inverses of matrices,} Proc. Cambridge Philos. Soc., 52 (1956), 600-601.}
\bibitem{shinozaki:1979}{Shinozaki, N.; Sibuya, M., \textit{Further results on the reverse-order law,} Linear Algebra Appl., 27 (1979), 9-16.}
\bibitem{udwadia:2020}{Udwadia, F.E.; Pennestri, E.; de Falco, D., \textit{Do all dual matrices have dual Moore–Penrose generalized inverses?,} Mech. Mach. Theory, 151 (2020), 103878.}
\bibitem{udwadia:2021}{Udwadia, F.E., When does a dual matrix have a dual generalized inverse?, Symmetry, 13 (2021), 1386.}
\bibitem{ulrich:2016}{Ulrich, M.; Steger, C., \textit{Hand-eye calibration of SCARA robots using dual quaternions,} Pattern Recognit.
Image Anal, 26 (2016), 231-239.}
\bibitem{wang:2023}{Wang, H.; Cui, C.; Liu, X., \textit{Dual r-rank decomposition and its applications,}
Comput. Appl. Math., 42 (2023), 349.}
\bibitem{wang:2021}{Wang, H., \textit{Characterizations and properties of the MPDGI and DMPGI,} Mech. Mach. Theory, 158 (2021), 104212.}
\bibitem{wang:2012}{Wang, X.; Yu, C.; Lin, Z., \textit{A dual quaternion solution to attitude and position control for rigid body
coordination,} IEEE Trans. Robot., 28 (2012), 1162-1170.}

\end{thebibliography}
\end{document}